\documentclass[12pt]{amsart}
\usepackage{amsfonts}
\usepackage{amssymb,amsmath,amsfonts,verbatim}
\usepackage{graphics}
\usepackage{epsfig}
 \usepackage{color}
\usepackage{enumerate}
\usepackage{mathrsfs}
\usepackage[mathscr]{euscript}

\usepackage{pslatex}

\newtheorem{thm}{Theorem}[section]

\newtheorem{lem}[thm]{Lemma}
\newtheorem{prop}[thm]{Proposition}
\theoremstyle{definition}

\theoremstyle{remark}

\numberwithin{equation}{section}

\newcommand{\mbb}{\mathbb}
\newcommand{\ra}{\rightarrow}
\newcommand{\z}{\zeta}
\newcommand{\pa}{\partial}
\newcommand{\ov}{\overline}
\newcommand{\sm}{\setminus}
\newcommand{\ep}{\epsilon}
\newcommand{\no}{\noindent}
\newcommand{\Om}{\Omega}
\newcommand{\cal}{\mathcal}
\newcommand{\ti}{\tilde}

\newcommand{\de}{\delta}
\newcommand{\ga}{\gamma}

\begin{document}
\title{Domains in complex surfaces with a noncompact automorphism group -- II}
\thanks{The author was supported in part by the DST SwarnaJayanti Fellowship 2009-2010.}
\author{Kaushal Verma}
\address{Department of Mathematics,
Indian Institute of Science, Bangalore 560 012, India}
\email{kverma@math.iisc.ernet.in}

\begin{abstract}
Let $X$ be an arbitrary complex surface and $D \subset X$ a domain that has a noncompact group of holomorphic automorphisms. A characterization of those domains $D$ that admit 
a smooth real analytic, finite type, boundary orbit accumulation point and whose closures are contained in a complete hyperbolic domain $D' \subset X$ is obtained.
\end{abstract}

\maketitle

\section{Introduction}
\no Let $D \subset \mbb C^n$, $n \ge 1$ be a bounded domain and let ${\rm Aut}(D)$ be the group of holomorphic automorphisms of $D$. There is a natural action of ${\rm Aut}(D)$ on 
$D$ given by
\[
(f, z) \mapsto f(z)
\]
where $f \in {\rm Aut}(D)$ and $z \in D$. Suppose the orbit of some point $p \in D$ under this action accumulates at $p_{\infty} \in \pa D$ -- call such a point a boundary 
orbit accumulation point. In this situation, it has been shown that (see \cite{BP1}--\cite{BP4}, \cite{Ber}, \cite{E}, \cite{Ki}, \cite{R} and \cite{W} among others) the 
nature of $\pa D$ near $p_{\infty}$ provides global information about $D$. The question of investigating this phenomenon when $D$ is a domain in a complex manifold was raised in 
\cite{CFKW} and \cite{GKK} and it was shown in the latter article that the Wong-Rosay theorem remains valid when $D$ is a domain in an arbitrary complex manifold with 
$p_{\infty} \in \pa D$ a strongly pseudoconvex point. In short, such a domain $D$ is biholomorphic to the unit ball $\mbb B^n \subset \mbb C^n$. Motivated by this result, 
it was shown in \cite{SV} that the analogues of \cite{BP1} and \cite{Ber} are also valid, with the same conclusion, when $D$ is a domain in an arbitrary complex surface and 
$p_{\infty}$ is a smooth weakly pseudoconvex point of finite type. The pseudoconvexity hypothesis near $p_{\infty}$ was dropped in \cite{BP4} and a local version of this 
result for bounded domains in $\mbb C^2$ and with the boundary $\pa D$ near $p_{\infty}$ being smooth real analytic and of finite type can be found in \cite{V}. The purpose of 
this article is to generalise the result in \cite{V} by finding all possible model domains when $D$ is a domain in an arbitrary complex surface $X$.

\begin{thm}
Let $X$ be an arbitrary complex surface and $D \subset X$ a domain. Suppose that $\ov D$ is contained in a complete hyperbolic domain $D' \subset X$ and that there exists a point 
$p \in D$ and a sequence $\{\phi_j\} \in {\rm Aut}(D)$ such that $\{\phi_j(p)\}$ converges to $p_{\infty} \in \pa D$. Assume that the boundary of $D$ is smooth real analytic 
and of finite type near $p_{\infty}$. Then exactly one of the following alternatives holds:  
 \begin{enumerate}
	\item[(i)] If $\dim{\rm Aut}(D) = 2$ then either
          \begin{itemize}
                \item $D \backsimeq \cal D_1 = \big\{ (z_1, z_2) \in \mbb C^2 :
        2 \Re z_2 + P_1(\Re z_1) < 0 \big\}$ where $P_1$ is a
        polynomial that depends on $\Re z_1$, or

                \item $D \backsimeq \cal D_2 = \big\{ (z_1, z_2) \in \mbb C^2 :
        2 \Re z_2 + P_2(\vert z_1 \vert^2) < 0 \big\}$ where $P_2$ is a 
	polynomial that depends on $\vert z_1 \vert^2$, or

                \item $D \backsimeq \cal D_3 =
        \big\{ (z_1, z_2) \in \mbb C^2 : 2 \Re z_2 +
        P_{2m}(z_1, \ov z_1) < 0 \big\}$ where $P_{2m}$ is a
        homogeneous polynomial of degree $2m$ without harmonic terms.
          \end{itemize}

        \item[(ii)] If $\dim{\rm Aut}(D) = 3$ then $D \backsimeq \cal D_4
        = \big\{ (z_1, z_2) \in \mbb C^2 : 2 \Re z_2 + (\Re
        z_1)^{2m} < 0 \big\}$ for some integer $m \ge 2$.

        \item[(iii)] If $\dim{\rm Aut}(D) = 4$ then $D \backsimeq \cal D_5
        = \big\{ (z_1, z_2) \in \mbb C^2 : 2 \Re z_2 +
        \vert z_1 \vert^{2m} < 0 \big\}$ for some integer $m \ge 2$.

        \item[(vi)] If $\dim{\rm Aut}(D) = 8$ then $D \backsimeq \cal D_6
        = \mbb B^2$ the unit ball in $\mbb C^2$.
  \end{enumerate}
The dimensions $0, 1, 5, 6, 7$ cannot occur with $D$ as above.
\end{thm}
\no To clarify several points, first note that $D$ is hyperbolic since it is contained in $D'$ which is assumed to be complete hyperbolic in the sense of Kobayashi. 
Therefore ${\rm Aut}(D)$ is a real Lie group endowed with the topology of uniform convergence on compact subsets of $D$. Moreover, the family $\phi_j : D \ra D 
\subset D'$ is normal since $D'$ is complete. By theorem 2.7 in \cite{TM} (which generalises Cartan's theorem  -- see \cite{N} pp. 78) we see that every possible limit map $\phi$ is 
either in ${\rm Aut}(D)$ or satisfies $\phi(D) \subset \pa D$. Since $\phi(p) = p_{\infty} \in \pa D$, it follows that $\phi(D) \subset \pa D$. Fix a neighbourhood $U$ of 
$p_{\infty}$ and a biholomorphism $\psi : U \ra \psi(U) \subset \mbb C^2$ such that $\psi(p_{\infty}) = 0$ and $\psi(U \cap \pa D)$ is a smooth real analytic hypersurface of finite 
type -- note that the type is a biholomorphic invariant and hence it suffices to work 
with a fixed, sufficiently small neighbourhood of $p_{\infty}$. Let $W$ be a neighbourhood of $p$ small enough so that $\phi(W) \subset U$. If 
possible, let $k > 0$ be the maximal rank of $\phi$ which is attained on the complement of an analytic set $A \subset D$. If $p \in W \sm A$, then the image of a small 
neighbourhood of $p$ that does not intersect $A$ under $\phi$ is a germ of a positive dimensional complex manifold contained in $U \cap \pa D$ and this is a contradiction. On 
the other hand if $p \in A$, pick $q \in W \sm A$ and repeat the above argument to see that $k = 0$ in this case as well. Thus $\phi(D) \equiv p_{\infty}$. Since this is true 
of any limit map, it follows that the entire sequence $\phi_j$ converges uniformly on compact subsets of $D$ to the constant map $\phi(z) \equiv p_{\infty}$. It follows that 
$D$ must be simply connected (see for example \cite{Kr}) for any loop $\gamma \subset D$ is contractible if and only if $\phi_j(\gamma)$ is so for all $j$. However, for all 
large $j$ the loop $\phi_j(\gamma) \subset U \cap D$ which is simply connected if $U$ is small enough. Hence $\phi_j(\gamma)$ is a trivial loop for large $j$ and hence so is 
$\gamma$. 

\medskip

Second, note that $\psi(p_{\infty})$ cannot belong to the envelope of holomorphy of $\psi(U \cap D)$. Indeed, for if not, then 
on the one hand we see from the above reasoning that the Jacobian determinant $\det (\psi \circ \phi_j)'$ must tend to zero uniformly on compact subsets 
of $D$. On the other hand, all the maps $\phi_j^{-1} \circ \psi^{-1} : \psi(U \cap D) \ra D \subset D'$ extend to a fixed, open neighbourhood of $\psi(p_{\infty})$ by a theorem of 
Ivashkovich (see \cite{Iv}) since $D'$ is complete. Moreover, the extensions of these maps near $\psi(p_{\infty})$ take values in $D'$. Hence there is an upper bound for $\det 
(\phi_j^{-1} \circ \psi^{-1})'$ near $\psi(p_{\infty})$ and this is a contradiction. As a consequence, this observation of Greene-Krantz is also valid in the situation of the main 
theorem.

\medskip

Third, recall the stratification of the smooth real analytic finite type hypersurface $U \cap \pa D$ that was used in \cite{V}. There is a 
biholomorphically invariant decomposition of $U \cap \pa D$ as the union of two relatively open sets, namely $\pa D^+$ (for brevity, we drop the reference to 
$\psi$) which consists of points near which 
$U \cap \pa D$ is pseudoconvex and $\hat D 
\cap \pa D$ that has those points which are in the envelope of holomorphy of $U \cap D$, and their closed complement $M_e$ which is a locally finite union of smooth real 
analytic arcs and points. Note that $M_e$ is contained in the set of Levi flat points $\cal L$ which by the finite type assumption is a codimension one real analytic subset of $U \cap 
\pa D$. By the second remark above, $p_{\infty} \notin \hat D \cap \pa D$. If $p_{\infty} \in \pa D^+$ then by \cite{SV} it follows 
that 
\[
D \backsimeq \big\{ (z_1, z_2) \in \mbb C^2 : 2 \Re z_2 + P_{2m}(z_1,\ov z_1) < 0 \big\}
\]
where $P_{2m}(z_1, \ov z_1)$ is a homogeneous subharmonic polynomial of degree $2m$ (this being the $1$-type of $U \cap \pa D$ near $p_{\infty}$) without harmonic terms. In this 
case, the assumption that $D \subset D'$ plays no role for pseudoconvexity of $U \cap \pa D$ near $p_{\infty}$ (an orbit accumulation point) is enough to guarantee that $D$ is 
complete hyperbolic -- see \cite{Ber}, \cite{Gau} for example. In particular, in the situation of the main theorem, the Levi form of $U \cap \pa D$ changes sign in every 
neighbourhood of 
$p_{\infty}$. Finally, a word about the assumption that $\ov D$ is contained in a complete hyperbolic domain $D' \subset X$. Perhaps the most natural assumption would be to {\it 
not} assume anything except finite type and smooth real analyticity of $U \cap \pa D$ near $p_{\infty}$. In this situation, the first thing to do would be to show the normality of 
$\cal O(\Delta, D)$, the family of holomorphic mappings from the unit disc $\Delta$ into $D$. And as in \cite{Ber} and \cite{Gau} this should be a consequence of 
understanding the rate of blow up of the Kobayashi metric on $D$ near $p_{\infty}$. That the metric can even be localised near $p_{\infty}$ near which the Levi form changes sign 
does not seem to be known. Therefore another possibility is to assume that $D$ is locally taut near $p_{\infty}$, i.e., $V \cap D$ is taut for some fixed neighbourhood 
$V$ of $p_{\infty}$. However, working with this also requires knowledge that an analytic disc $f : \Delta \ra D$ with $f(0)$ close to $p_{\infty}$ can be localised. 
Moreover, if we strengthen the hypothesis on $D$ by assuming that it is complete hyperbolic, then $D$ would be pseudoconvex near $p_{\infty}$. The model domains in this case have been 
determined in \cite{SV}. With these observations a plausible 
hypothesis seemed that of requiring that $\ov D \subset D'$ where $D' \subset X$ is complete -- and this, though being global in nature, seemed to complement well the assumption 
made in \cite{V} that $D \subset \mbb C^2$ is a bounded domain.

\medskip

The general strategy is the same as in \cite{V}. Note that since $D$ is hyperbolic it follows from \cite{Ka}, \cite{Ko} that $0 \le \dim {\rm Aut}(D) \le n^2 + 2n = 8$ as $n = 
2$. Furthermore by \cite{Ka} it is known that if $\dim {\rm Aut}(D) \ge 5$, then $D$ is homogeneous and hence there is an orbit that clusters at strongly pseudoconvex points in 
$U \cap \pa D$. Such points form a non-empty open subset of $U \cap \pa D$ that contains $p_{\infty}$ in its closure and this follows from the decompositon of $U \cap \pa D$ 
alluded to above. Consequently by \cite{GKK}, $D \backsimeq \mbb B^2$. Therefore it suffices to treat the case when $0 \le \dim {\rm Aut}(D) \le 4$. An initial scaling of $D$ 
using the orbit $\{\phi_j(p)\}$ as described below shows that $D$ is biholomorphic to a model domain of the form
\[
G = \big\{ (z_1, z_2) \in \mbb C^2 : 2 \Re z_2 + P(z_1, \ov z_1) < 0 \big\}
\]
where $P(z_1, \ov z_1)$ is a polynomial without harmonic terms. Let $g : G \ra D$ be the biholomorphism. $G$ is evidently invariant under the one parameter subgroup of 
translations in the imaginary $z_2$-direction, i.e., $T_t(z_1, z_2) = (z_1, z_2 + it)$ for $t \in \mbb R$. This shows that $\dim {\rm Aut}(D) \ge 1$. 
If $\dim {\rm Aut}(D) = \dim {\rm Aut}(G) = 1$, it is possible to explicitly write down what can element of ${\rm Aut}(G)$ should look like and this description shows that the 
orbits in $G$ stay uniformly away from the boundary and accumulate only at the point at infinity in $\pa G$. Using the assumption that $D$ is contained in a taut domain, it can be 
seen that the Kobayashi metric in $D$ blows up near $p_{\infty}$. Let $\cal X = g_{\ast}(i \pa / \pa z_2)$; note that $i \pa / \pa z_2$ is a holomorphic vector field in $G$ whose 
real part generates the translations $T_t$. Then $p_{\infty}$ is seen to be an isolated zero of $\cal X$ on $\pa D$ and the arguments of \cite{BP4} show that $\cal X$ must be 
parabolic and this forces $D$ to be equivalent to an ellipsoid whose automorphism group is four dimensional. This is a contradiction. When $\dim {\rm Aut}(D) = 2$, two cases arise 
depending on whether ${\rm Aut}(D)^c$, the connected component of the identity in ${\rm Aut}(D)$ is abelian or not. In the former case, ${\rm Aut}(D)^c$ must be isomorphic to 
either $\mbb R^2$ or to $\mbb R \times \mbb S^1$. These lead to the conclusion that $D \backsimeq \cal D_1$ or $D \backsimeq \cal D_2$. In the non-abelian case ${\rm Aut}(D)^c$ 
is solvable and it can be shown that $D \backsimeq \cal D_3$. A case-by-case analysis is used when $\dim {\rm Aut}(D) = 3, 4$ to identify the relevant domain from the 
classification obtained by A. V. Isaev in \cite{I1}, \cite{I2}. While the argument remains the same in some cases, we take this opportunity to streamline and provide alternate 
proofs in some instances -- for example, ruling out the possibility that $\dim {\rm Aut}(D) = 1$ and identifying the right model domain when $\dim {\rm Aut}(D) = 3$. There are several 
possibilities in \cite{I1} and here we focus on three interesting classes from that list, as the proof for the others remains the same. Nothing changes when $\dim {\rm Aut}(D) = 2, 
4$, i.e., the same proofs from \cite{V} carry over to these cases and we have decided to be brief, the emphasis being not to merely repeat what carries over to this situation 
from \cite{V}, but to identify and focus on the differences instead.


\section{The dimension of ${\rm Aut}(D)$ is at least two}

\no To describe the scaling of $D$ using the base point $p$ and the sequence $\{ \phi_j \} \in {\rm Aut}(D)$, first note that for $j$ large, there is a unique point $\ti p_j \in 
\psi(U \cap \pa D)$ such that
\begin{equation}
{\rm dist}(\psi \circ \phi_j(p), \psi(U \cap \pa D)) = \vert \ti p_j - \psi \circ \phi_j(p) \vert.
\end{equation}
By a rotation of coordinates, we may assume that the defining function $\rho(z)$ for $\psi(U \cap \pa D)$ is of the form
\[
\rho(z) = 2 \Re z_2 + \sum_{k, l} c_{kl}(y_2) z_1^k \ov z_1^l
\]
where $c_{00}(y_2) = O(y_2^2)$ and $c_{10}(y_2) = \ov c_{01}(y_2) = O(y_2)$. Let $m$ be the type of $\psi(U \cap \pa D)$ at the origin. By definition, there exist $k, l$ both at 
least one and $k + l = m$ for which $c_{kl}(0) \not= 0$ and $c_{kl}(0) = 0$ for all other $k + l < m$. The pure terms, if any, up to order $m$ in the defining function can be 
removed by a polynomial automorphism of the form
\begin{equation}
(z_1, z_2) \mapsto \big( z_1, z_2 + \sum_{k \le m} (c_{k0}(0)/2) z_1^k \big).
\end{equation}
These coordinate changes will be absorbed in $\psi$. Let $\psi^j_{p, 1}(z) = z - \ti p_j$ so that $\psi^j_{p, 1}(\ti p_j) = 0$. A unitary rotation $\psi^j_{p, 2}(z)$ then ensures 
that the outer real normal to $\psi^j_{p, 1} \circ \psi(U \cap \pa D)$ at the origin is the real $z_2$-axis. The defining function for $\psi^j_{p, 2} \circ \psi^j_{p, 1} \circ 
\psi(U \cap \pa D)$ near the origin is then of the form
\begin{equation}
\rho_j(z) = 2 \Re z_2 + \sum_{k, l} c^j_{kl}(y_2) z_1^k \ov z_1^l
\end{equation}
with the same normalisations on the coefficients $c^j_{00}(y_2)$ and $c^j_{10}(y_2)$ as described above. Since $\ti p_j \ra 0$ it follows that both $\psi^j_{p, 1}$ and $\psi^j_{p, 
2}$ converge to the identity uniformly on compact subsets of $\mbb C^2$. Note that the type of $\psi^j_{p, 2} \circ \psi^j_{p, 1} \circ \psi(U \cap \pa D)$ is at most $m$ for all 
$j$ and an automorphism of the form (2.2) can be used to remove all pure terms up to order $m$ from $\rho_j(z)$. Denote this by $\psi^j_{p, 3}$. Lastly, note that $\psi \circ 
\phi_j(p)$ is on the inner real normal to $\psi(U \cap \pa D)$ at $\ti p_j$ and it follows that $\psi^j_{p, 2} \circ \psi^j_{p, 1} \circ \psi \circ \phi_j(p) = (0, -\de_j)$ for 
some $\de_j > 0$ and the explicit form of (2.2) shows that this is unchanged by $\psi^j_{p, 3}$. Let 
\[
\psi^j_{p, 4}(z_1, z_2) = (z_1/\ep_j, z_2/\de_j)
\]
where $\ep_j > 0$ will be chosen in the next step. The defining function for $\psi^j_p \circ \psi (U \cap \pa D)$ near the origin, where $\psi^j_p = \psi^j_{p, 4} \circ \psi^j_{p, 
3} \circ \psi^j_{p, 2} \circ \psi^j_{p, 1}$, is given by
\[
\rho_{j, p}(z) = \de_j^{-1} \; \rho_j(\ep_j z_1, \de_j z_2) = 2 \Re z_2 + \sum_{k, l} \ep_j^{k + l} \de_j^{-1} c^j_{kl}(\de_j y_2) z_1^k \ov z_1^l.
\]
Observe that $\psi^j_p \circ \psi \circ \phi_j(p) = (0, -1)$ for all $j$. Now choose $\ep_j > 0$ by demanding that
\[
\max \big\{ \vert \ep_j^{k + l} \de_j^{-1} c^j_{kl}(0) \vert : k + l \le m \big\} = 1
\]
for all $j$. In particular, note that $\{ \ep_j^m \de_j^{-1} \}$ is bounded and by passing to a subsequence it follows that 
\[
\rho_{j, p}(z) \ra \rho_p = 2 \Re z_2 + P(z_1, \ov z_1)
\]
in the $C^{\infty}$ topology on compact subsets of $\mbb C^2$, where $P(z_1, \ov z_1)$ is a polynomial of degree at most $m$ without any harmonic terms. Therefore 
the domains $G_{j, p} = \psi^j_p \circ \psi(U \cap D)$ converge to
\[
G_p = \{ (z_1, z_2) \in \mbb C^2 : 2 \Re z_2 + P(z_1, \ov z_1) < 0\}
\]
in the Hausdorff sense. Let $K \subset G_p$ be a relatively compact domain containing the base point $(0, -1)$. Then $K \subset \psi^j_p \circ \psi(U \cap D)$ for all large 
$j$ and therefore the mappings
\[
g^j_p : (\psi^j_p \circ \psi \circ \phi_j)^{-1} : K \ra D \subset D'
\]
are well defined and satisfy $g^j_p(0, -1) = p$. The completeness of $D'$ shows that the family $\{g^j_p\}$ is normal and hence there is a holomorphic limit $g_p : G_p \ra \ov D$ 
with $g_p(0, -1) = p$. It remains to show that $g_p$ is a biholomorphism from $G_p$ onto $D$. For this, recall the observation made in \cite{BP4} that since $P(z_1, \ov z_1)$ is 
not harmonic, the envelope of holomorphy of $G_p$ is either all of $\mbb C^2$ or $\pa G_p$ contains a strongly pseudoconvex point. The former situation cannot hold  -- indeed, by 
\cite{Iv} again, the map $g_p$ will extend to $\mbb C^2$ taking values in $D'$ and since $D'$ is complete, $g_p(z) \equiv p$. Let $W \subset \mbb C^2$ be a bounded domain that 
intersects infinitely many of the boundaries $\psi^j_p \circ \psi(U \cap \pa D)$ -- and hence also $\pa G_p$. Then for each $j$, note that the cluster set of $W \cap \psi^j_p 
\circ \psi(U \cap \pa D)$ under $g^j_p$ is contained in $\pa D$ since $\phi_j \in {\rm Aut}(D)$. Now, if the envelope of $G_p$ were all of $\mbb C^2$, it is possible to 
find a domain $\Om$ with $\Om \cap \pa G_p \not= \emptyset$ on which the family $\{g^j_p\}$ would converge uniformly. In this case, by passing to the limit, we see that $g_p(U 
\cap \pa G_p) \subset \pa D$ and thus $g_p$ cannot be the constant map. Therefore there must be a strongly pseudoconvex point, say $\z$ on $\pa G_p$. Fix $r > 0$ so that all 
points on $\pa G_p \cap B(\z, r)$ are strongly pseudoconvex and since $\rho_{j, p} \ra \rho_{\infty, p}$ in the $C^{\infty}$ topology on $B(\z, r)$, it follows that each of the 
open pieces $\psi^j_p \circ \psi(U \cap \pa D) \cap B(\z, r)$ are themselves strongly pseudoconvex for $j \gg 1$. For a complex manifold $M$, let $F_M(z, v)$ denote the 
Kobayashi metric at $z \in M$ along a tangent vector $v$ at $z$. By the stability of the Kobayashi metric under smooth strongly pseudoconvex perturbations, it follows that for 
all $q \in B(\z, r) \cap G_p$
\[
F_{G_{j, p}}(q, v) \ge c \vert v \vert
\]
for some uniform $c > 0$ and by the invariance of the Kobayashi metric we see that
\begin{equation}
F_{\phi_j^{-1}(U \cap D)}(g^j_p(q), dg^j_p(q) v) = F_{G_{j, p}}(q, v) \ge c \vert v \vert.
\end{equation}
Since the automorphisms $\phi_j \ra p_{\infty}$ uniformly on compact subsets of $D$, it can be seen that the domains $\phi_j^{-1}(U \cap D)$ form an exhaustion of $D$ in the 
sense that for any compact $K \subset D$, there is an index $j_0$ for which $K \subset \phi_{j_0}^{-1}(U \cap D)$. Furthermore, as $g_p(0, -1) = p \in D$, it follows that 
$g_p^{-1}(\pa D)$ is closed and nowhere dense in $G_p$ and therefore it is possible to choose a $q \in (B(\z, r) \cap G_p) \sm g_p^{-1}(\pa D)$. This ensures that $g^j_p(q) \ra 
g_p(q) \in D$. Now, the completeness of $D'$ implies that the Kobayashi metrics on $\phi_j^{-1}(U \cap D)$ converge to the corresponding metric on $D$ and thus (2.4) shows that
\begin{equation}
F_D(g_p(q), dg_p(q) v) \gtrsim c \vert v \vert.
\end{equation}
Thus $dg_p(q)$ has full rank. Thus the rank of $dg_p$ can be smaller only on an analytic set $A \subset G_p$ of dimension at most one. Pick $\ti q \in A$ and let $N_1, N_2$ be 
small neighbouroods of $\ti q$ and $g_p(\ti q)$ respectively such that $g^j_p(N_1) \subset N_2$ for $j \gg 1$. By identifying $N_2$ with an open subset of $\mbb C^2$, 
Hurwitz's theorem applied to the Jacobians $\det (dg^j_p)$ shows that either $\det (dg^j_p)$ never vanishes or is identically zero in $N_1$. Since $A$ has strictly smaller 
dimension it follows that $dg_p$ has full rank everywhere, i.e., $A$ must be empty. Hence $g_p$ is locally biholomorphic in $G_p$ and therefore $g_p(G_p) \subset D$. Injectivity 
of $g_p$ is now a consequence of the fact that $g^j_p$ are all biholomorphic and they converge uniformly on compact subsets of $G_p$ to $g_p$. 

\medskip

To conclude, we have to show that 
$D_p = g_p(G_p)$ is all of $D$. If not, pick $\ti p \in \pa D_p \cap D$ and note that since $\phi_j(\ti p) \ra p_{\infty}$, the scaling argument above can be repeated to get a 
biholomorphism $g_{\ti p} : G_{\ti p} \ra g_{\ti p}(G_{\ti p}) \subset D$. Here $G_{\ti p}$ has the same form as $G_p$ with possibly a different polynomial than $P(z_1, \ov z_1)$. 
Note that $V = D_p \cap D_{\ti p}$ is then a nonempty open subset of $D$. Let $f^j_p = (g^j_p)^{-1}, f_p = g_p^{-1}$ and $f^j_{\ti p} = (g^j_{\ti p})^{-1}, f_{\ti p} = g_{\ti 
p}^{-1}$. Observe that both $f_p, f_{\ti p}$ are biholomorphic on $V$, and that both $f^j_p, f^j_{\ti p}$ are defined on a given compact set in $D$ for large $j$. We may 
write $f_p = A \circ f_{\ti p}$ where $A = g_p^{-1} \circ g_{\ti p}$ is biholomorphic on $f_{\ti p}(V)$. But more can be said about $A$ -- indeed, by definition we have
\[
g^j_p \circ \psi^j_p \circ (\psi^j_{\ti p})^{-1} = g^j_{\ti p}
\]
where $A_j = \psi^j_p \circ (\psi^j_{\ti p})^{-1}$ are polynomial automorphisms of $\mbb C^2$ of bounded degree as their construction shows. Since $g^j_p$ and $g^j_{\ti p}$ 
converge to $g_p$ and $g_{\ti p}$ respectively, we may take $A$ as the limit of $A_j$ on $f_{\ti p}(V)$ and conclude that $A$ is also a 
polynomial automorphism of $\mbb C^2$. Now the functional equation $f_p = A \circ f_{\ti p}$ extends $f_p$ as a biholomorphic mapping from a small neighbourhood $W$ of $\ti p$ 
onto $W'$, a neighbourhood of $f_p(\ti p)$. On the other hand, note first that since $g_{\ti p}$ is biholomorphic near $(0, -1)$ and maps it to $\ti p$, it follows that $f^j_{\ti 
p}$ form a normal family on $W$, after possibly shrinking it if necessary. As a consequence, the equality 
\[
f^j_p = A_j \circ f^j_{\ti p}
\]
which holds on $W$ for $j$ large, shows that $f^j_p$ converges to $f_p$ on $W$ and hence in the limit we see that $f_p(W) \subset \ov G_p$. That is, $f_p(W)$ cannot contain a 
neighbourhood of $f_p(\ti p)$ which is a contradiction. Hence $g_p : G_p \ra D$ is biholomorphic and since $G_p$ is invariant under the translations $T_t$, it follows that 
$\dim {\rm Aut}(D) = \dim {\rm Aut}(G_p) \ge 1$. In the sequel, we will write $g, G$ in place of $g_p, G_p$ respectively.

\medskip

\no Recall that $p_{\infty}$ is not in the envelope of holomorphy of $U \cap D$ where $(U, \phi)$ is the coordinate chart around $p_{\infty}$ that was fixed earlier. Let 
$\Delta \subset \mbb C$ be the unit disc. The following estimate on the Kobayashi metric near $p_{\infty}$ will be useful.

\begin{lem}
For every $r \in (0, 1)$, there is a neighbourhood $V$ of $p_{\infty}$ compactly contained in $U$ such that every analytic disc $f : \Delta \ra D$ with $f(0) \in V$ satisfies 
$f(r \Delta) \subset U$. As a result, the Kobayashi metric can be localised near $p_{\infty}$ -- there is a constant $C > 0$ such that
\[
C \cdot F_{U \cap D}(p, v) \le F_D(p, v) \le  F_{U \cap D}(p, v)
\]
uniformly for all $p \in V \cap D$ and tangent vectors $v$ at $p$. Moreover,
\[
F_D(p, v) / \vert v \vert \ra \infty
\]
as $p \ra p_{\infty}$. In particular, for any neighbourhood $V$ of $p_{\infty}$ and $R < \infty$, there exists another neighbourhood $W \subset V$ of 
$p_{\infty}$ such that the Kobayashi ball $B^k_D(p, R) \subset V$ whenever $p \in W \cap D$.
\end{lem}

\begin{proof}
Let $f_{\nu} : \Delta \ra D \subset D'$ be a sequence of holomorphic disks with $f_{\nu}(0) = p_{\nu} \ra p_{\infty}$. The completeness of $D'$ implies that some subsequence of 
$\{ f_{\nu} \}$ converges uniformly on compact subsets of $\Delta$ to a holomorphic limit $f : \Delta \ra \ov D$ and $f(0) = p_{\infty}$. Suppose that $f(z) \not\equiv 
p_{\infty}$ on $\Delta$. Let $\eta > 0$ be such that $f(\eta \Delta) \subset U$. Since $U \cap \pa D$ is of finite type, no open subset of $f(\eta \Delta)$ can be contained in 
it and hence $f(\eta \Delta) \cap D \not= \emptyset$. By the strong disk theorem (\cite{Vl}) it follows that $p_{\infty}$ belongs to the envelope of holomorphy of $U \cap D$ 
which is a contradiction. Therefore $f(z) \equiv p_{\infty}$ and this shows that all limit functions for the given family of holomorphic disks are constant. The first claim 
follows and the equivalence of the metrics on $U \cap D$ and $D$ is then a consequence of the definition of the Kobayashi metric. 

\medskip

If there exists a sequence $p_{\nu} \ra p_{\infty}$ and non-zero vectors $v_{\nu}$ at $p_{\nu}$ and a constant $C$ such that $F_D(p_{\nu}, v_{\nu}) \le C \vert v_{\nu} \vert$, 
then there would exist a uniform $r > 0$ and holomorphic disks $f_{\nu} \in \cal O(r \Delta, D)$ with $f_{\nu}(0) = p_{\nu}$ and $df_{\nu}(0) = v_{\nu}$. By the homogeneity 
of the metric in the vector variable, we may assume that $\vert v_{\nu} \vert = 1$ for all $\nu$. The argument above shows that every possible limit function $f$ of the 
family $\{ f_{\nu} \}$ is constant which contradicts $\vert df(0) \vert = 1$. Therefore $F_D(p, v) / \vert v \vert$ blows up as $p \ra p_{\infty}$.

\medskip

For the claim about the size of $B^k_D(p, R)$, let us work in local coordinates around $\phi(p_{\infty}) = 0$. For $a, b \in U \cap D$, let $d(a, b)$ denote the euclidean 
distance on $U \cap D$ induced by $\phi$. For a given neighbourhood $V$ of $p_{\infty}$ and $R < \infty$, 
let $p_{\infty} \in N_2 \subset V$ be such that $F_D(p, v) / \vert v \vert \ge 2 R$ for all $p \in N_2 \cap D$ and tangent vectors $v$ at $p$. We may assume without loss of 
generality that $N_2 \subset U$ and $N_2 = \phi^{-1}(B(0, 2))$. Let $N_1 = \phi^{-1}(B(0, 1))$. Fix $p \in N_1 \cap D$ and $q \in D$ and let $\gamma(t)$ be a path 
in $D$ parametrised by $[0, 1]$ with $\gamma(0) = p$ and $\gamma(1) = q$ such that
\[
\int_0^1 F_D(\ga(t), \ga'(t)) \; dt \le d^k_D(p, q) + \ep
\]
where $\ep > 0$ is given and $d^k_D(p, q)$ is the Kobayashi distance between $p, q$. Suppose that $q \in N_1 \cap D$; two cases now arise -- first, if the entire path $\gamma 
\subset N_2 \cap D$, then 
\[
2R \; d(p, q) \le \int_0^1 F_D(\ga(t), \ga'(t)) \; dt \le d^k_D(p, q) + \ep.
\]
Second, if $\gamma$ does not entirely lie in $N_2 \cap D$, then there is a connected component of $\gamma$ that contains $p$ and a point $a \in \pa N_2 \cap D$. The length of 
this connected component is at least $2 \ge d(p, q)$. On the other hand, if $q \in D \sm N_1$, then the length of this path can be bounded from below by simply $2R$. Thus we 
get
\[
d^k_D(p, q) \ge 2R \; d(p, q) - \ep
\]
if $q \in N_1 \cap D$ and $d^k_D(p, q) \ge 2 R$ otherwise. Now if $p \in N_1 \cap D$ and $q \in B^k_D(p, R)$, it follows from these comparisons that $q \in N_1 \cap D$ which 
completes the proof.
\end{proof}

\no The holomorphic vector field $\cal X = g_{\ast}(i \pa / \pa z_2)$ on $D$ is such that its real part $\Re \cal X = (\cal X + \ov{\cal X}) / 2$ generates the one parameter 
subgroup $L_t = g \circ T_t \circ g^{-1} = \exp(t \; \Re \cal X) \in {\rm Aut}(D)$. Two observations can be made about $\cal X$ at this stage -- first, Proposition 2.3 of \cite{V} 
shows that $(L_t)$ induces a local one parameter group of holomorphic automorphisms of a neighbourhood of $p_{\infty}$ when $D \subset \mbb C^2$ is a bounded domain. 
In particular, $\cal X$ extends as a holomorphic vector field near $p_{\infty}$. The proof of this relies on a local parametrised version of the reflection principle from 
\cite{DP}, the main tools being the use of Segre varieties and their invariance property under biholomorphisms to construct the desired extension of $(L_t)$ near $p_{\infty}$ 
for all $\vert t \vert < \eta$ for a fixed $\eta > 0$. The same arguments can be applied in the local coordinates induced on $U$ by $\phi$ to get the same conclusion in the 
setting of the main theorem as well. Second, consider the pullback of the orbit $\{\phi_j(p)\} \in D$ under the equivalence $g : G \ra D$, i.e., let $g^{-1} \circ \phi_j(p) = (a_j, b_j) 
\in G$ and 
\[
2 \ep_j = 2 \Re b_j + P(a_j, \ov a_j). 
\]
Note that $\ep_j < 0$ for all $j$. Proposition 2.5 of \cite{V} shows that if $\vert \ep_j \vert > c > 0$ for all large $j$, then 
$\cal X$ vanishes to finite order at $p_{\infty}$. The proof of this uses the boundedness of $D \subset \mbb C^2$ which in particular implies that a family of holomorphic maps into $D$ 
is normal. The same argument can be applied in the situation of the main theorem since $D \subset D'$ and $D'$ is assumed to be complete hyperbolic. Thus we have:

\begin{prop}
The group $(L_t)$ induces a local one parameter group of holomorphic automorphisms of a neighbourhood of $p_{\infty}$ in $X$. In particular, $\cal X$ extends as a holomorphic 
vector field near $p_{\infty}$. Moreover, if $\vert \ep_j \vert > c > 0$ for all large $j$, then $\cal X$ vanishes to finite order at $p_{\infty}$.
\end{prop}

\no The next step is to describe what the elements of ${\rm Aut}(G)$ look like under the assumption that $\dim {\rm Aut}(G) = 1$. This calculation was done in Propositions 2.6 and 2.7 
of \cite{V} and they remain valid here since they do not involve any features of $D$. The conclusion is that if $g \in {\rm Aut}(G)$ then
\[
g(z_1, z_2) = (g_1(z_1, z_2), g_2(z_1, z_2)) = (\alpha z_1 + \beta, \phi(z_1) + a z_2)
\]
where $\vert \alpha \vert = 1, a = \pm 1, \beta \in \mbb C$ and $\phi(z_1)$ is a holomorphic polynomial. Moreover, if $q = (q_1, q_2) \in G$, $g \in {\rm Aut}(G)$ and 
\[
E = 2 \Re(g_2(q_1, q_2)) + P(g_1(q_1, q_2), \ov{g_1(q_1, q_2)})
\]
then $\vert E \vert = \vert 2 \Re q_2 + P(q_1, \ov q_1) \vert$ as Lemma 2.8 of \cite{V} shows. Hence $\vert E \vert$ is independent of $g$.

\begin{prop}
The dimension of ${\rm Aut}(D)$ is at least two.
\end{prop}

\begin{proof}
Suppose that $\dim {\rm Aut}(D) = \dim {\rm Aut}(G) = 1$. Write
\[
(a_j, b_j) = g^{-1} \circ \phi_j(p) = g^{-1} \circ \phi_j \circ g(g(p))
\]
and note that $g^{-1} \circ \phi_j \circ g \in {\rm Aut}(G)$ for all $j \ge 1$. Let $g^{-1}(p) = q = (q_1, q_2) \in G$. By the arguments summarized above, it follows that
\begin{equation}
\vert 2 \Re b_j + P(a_j, \ov a_j) \vert = \vert 2 \Re q_2 + P(q_1, \ov q_1) \vert > 0
\end{equation}
for all $j \ge 1$. This shows that the orbit $\{g^{-1} \circ \phi_j(p)\} \in G$ can only cluster at the point at infinity in $\pa G$. Let
\[
\eta =  \vert 2 \Re q_2 + P(q_1, \ov q_1) \vert > 0
\]
and for $r > 0$ define
\[
G_r = \big\{ (z_1, z_2) \in \mbb C^2 : 2 \Re z_2 + P(z_1, \ov z_1) < -r \big\} \subset G.
\]
Observe that the boundaries of $G$ and $G_r$ intersect only at the point at infinity for all $r > 0$. Furthermore, the entire orbit $(a_j, b_j)$ and $q$ are contained in $G_{\eta/2}$ 
by (2.6). By Proposition 2.2 above it follows that $\cal X(p_{\infty}) = 0$ and by Lemma 3.5 of \cite{BP4} the intersection of the zero set of $\cal X$ with $\pa D$ contains 
$p_{\infty}$ as an isolated point. Now regard $g$ as a holomorphic mapping from $G_{\eta/2}$ into $D$. The sequence $(a_j, b_j) \in G_{\eta/2}$ converges to the point at infinity in 
$\pa G_{\eta/2}$ and its image under $g$, namely $\phi_j(p)$, converges to $p_{\infty}$. Proposition 2.2 also shows that if the cluster set of the point at infinity in $\pa G_{\eta/2}$ 
intersects $\pa D$ near $p_{\infty}$, then the vector field $\cal X$ vanishes at all such points. Since the cluster set of the point at infinity in $\pa G_{\eta/2}$ under $g$ is 
connected and contains $p_{\infty}$ as an isolated point, it must equal $p_{\infty}$.

\medskip

Thus for a given small neighbourhood $U$ of $p_{\infty}$ there exists a neighbourhood of the point at infinity in $\pa G_{\eta/2}$ which is mapped by $g$ into $U \cap D$. However, 
a neighbourhood of infinity in $G_{\eta/2}$ contains $G_M$ for some large $M > 0$. Fix a point $s \in g(G_{M}) \subset D$ and let $\tilde s = (\tilde s_1, \tilde 
s_2) \in G_{M}$ be such that $s = g(\tilde s)$. Note that
\[
L_t(s) = L_t \circ g(\tilde s) = g \circ T_t(\tilde s_1, \tilde s_2) = g(\tilde s_1, \tilde s_2 + it)
\]
which gives $L_t(s) \ra p_{\infty}$ as $\vert t \vert \ra \infty$. For any compact $K \subset D$ there exists $R > 0$ such that $K$ is contained in the Kobayashi ball $B^k_D(s, R)$. 
Hence $L_t(K) \subset B^k_D(L_t(s), R)$ for any $t \in \mbb R$. By Lemma 2.1 it follows that $L_t$ moves any point in $D$ in both forward and backward time to $p_{\infty}$, i.e., the 
action of $L_t$ on $D$ is parabolic. The arguments of \cite{BP4} can now be applied to show that 
\[
D \backsimeq \big\{ (z_1, z_2) \in \mbb C^2 : \vert z_1 \vert^2 + \vert z_2 \vert^{2m} < 1 \big\}
\]
for some integer $m \ge 1$. Thus $\dim {\rm Aut}(D) = 4$ which is a contradiction.
\end{proof}

\no In case $\dim {\rm Aut}(D) = 2$, note that the calculations done in section 3 of \cite{V} deal with only the defining function of $G$ and hence they apply in this situation as well. 
Indeed, the following dichotomy holds -- here ${\rm Aut}(D)^c$ is the connected component of the identity.

\begin{enumerate}

\item[(i)] If ${\rm Aut}(D)^c$ is abelian, then $D$ is biholomorphic to either
\[
\cal D_1 = \big\{ (z_1, z_2) \in \mbb C^2 : 2 \Re z_2 + P_1(\Re z_1) < 0 \big\}
\]
or 
\[
\cal D_2 = \big\{ (z_1, z_2) \in \mbb C^2 : 2 \Re z_2 + P_2(\vert z_1 \vert^2) < 0 \big\}
\]
for some polynomials $P_1, P_2$ that depend only on $\Re z_1$ or $\vert z_1 \vert^2$ respectively.

\item[(ii)] If ${\rm Aut}(D)^c$ is non-abelian then $D$ is biholomorphic to
\[
\cal D_3 = \big\{ (z_1, z_2) \in \mbb C^2 : 2 \Re z_2 + P_{2m}(z_1, \ov z_1) < 0 \big\}
\]
where $P_{2m}(z_1, \ov z_1)$ is a homogeneous polynomial of degree $2m$ without harmonic terms.
\end{enumerate}


\section{Model domains when ${\rm Aut}(D)$ is three dimensional}

\subsection{A tube domain and its finite and infinite sheeted covers} For $0 \le s < t < \infty$ define
\[
\mathfrak S_{s, t} = \big\{ (z_1, z_2) \in \mbb C^2 : s < (\Re z_1)^2 + (\Re z_2)^2 < t \big\}
\]
which is a non-simply connected tube domain over a nonconvex base. Evidently $D$ cannot be biholomorphic to $\mathfrak S_{s, t}$ since $D$ is simply connected as observed earlier. It is 
possible to consider finite and infinite sheeted covers of $\mathfrak S_{s, t}$. To obtain a finite sheeted cover, consider the $n$-sheeted covering self map 
\[
\Phi_{\cal \chi}^{(n)} : \mbb C^2 \sm \big\{ \Re z_1 = \Re z_2 = 0 \big\} \ra  \mbb C^2 \sm \big\{ \Re z_1 = \Re z_2 = 0 \big\} 
\]
whose components are given by

\begin{align*}
\ti z_1 &= \Re \big( (\Re z_1 + i \Re z_2)^n \big) + i \Im z_1,\\
\ti z_2 &= \Im \big( (\Re z_1 + i \Re z_2)^n \big) + i \Im z_2.   
\end{align*}

\no Equip $\mbb C^2 \sm \{\Re z_1 = \Re z_2 = 0\}$ with the pull-back complex structure using $\Phi_{\cal \chi}^{(n)}$ and call the resulting complex surface $M_{\cal \chi}^{(n)}$. For 
$0 \le s < t < \infty$ and $n \ge 2$ define
\[
\mathfrak S^{(n)}_{s, t} = \big\{ (z_1, z_2) \in M_{\cal \chi}^{(n)} : s^{1/n} < (\Re z_1)^2 + (\Re z_2)^2 < t^{1/n} \big\}.
\]
Then $\Phi_{\cal \chi}^{(n)}$ is an $n$-sheeted holomorphic covering map from $\mathfrak S^{(n)}_{s, t}$ onto $\mathfrak S_{s, t}$. It is clear that the domains $\mathfrak S^{(n)}_{s, 
t}$ are not simply connected and hence $D$ cannot be equivalent to any of them. Proposition 4.7 in \cite{V} provides a different proof of this fact which uses ideas that are applicable 
for other classes of domains as well. This can be adapted in the setting of theorem 1.1 as follows:

\begin{prop}
There cannot exist a proper holomorphic mapping from $D$ onto $\mathfrak S_{s, t}$ for all $0 \le s < t < \infty$. In particular, $D$ cannot be equivalent to $\mathfrak S^{(n)}_{s, t}$ 
for any $n \ge 2$ and $0 \le s < t < \infty$.
\end{prop}

\begin{proof}
Let $\pi : D \ra \mathfrak S_{s, t}$ be a proper holomorphic mapping. The case when $0 < s < t < \infty$ will be considered first. The boundary of $\mathfrak S_{s, t}$ has two 
components, namely
\begin{align*}
\pa \mathfrak S_{s, t}^+ &= \big\{ (z_1, z_2) \in \mbb C^2 : (\Re z_1)^2 + (\Re z_2)^2 = t \big\}, \; {\rm and} \\
\pa \mathfrak S_{s, t}^- &= \big\{ (z_1, z_2) \in \mbb C^2 : (\Re z_1)^2 + (\Re z_2)^2 = s \big\}.
\end{align*}
The orientation induced on these pieces by $\mathfrak S_{s, t}$ makes them strongly pseudoconvex and strongly pseudoconcave respectively. Lemma 2.1 of \cite{V} shows that there is a two 
dimensional stratum $S \subset \cal L \cap \hat D$ that clusters at $p_{\infty}$ -- this is a purely local assertion and hence it remains valid here as well. Pick $a \in S$ near 
$p_{\infty}$ and let $W$ be a small neighbourhood of $a$ so that $\pi$ extends holomorphically to $W$. Note that $(W \cap \pa D) \sm S$ consists of points that are either strongly 
pseudoconvex or strongly pseudoconcave. Let $V_{\pi} \subset W$ be the branching locus of $\pi : W \ra \mbb C^2$. Since $\pa D$ is finite type, it follows that $V_{\pi} \cap \pa D$ has 
real dimension at most one. There are two possibilities now -- first, if $\pi(a) \in \mathfrak S_{s, t}^+$, then choose a strongly pseudoconcave point $a' \in (W \cap \pa D) \sm 
V_{\pi}$. Thus $\pi$ maps a neighbourhood of $a'$, which is strongly pseudoconcave, locally biholomorphically onto a neighbourhood of  $\pi(a') \in \pa \mathfrak S_{s, t}^+$ and this is a 
contradiction. A similar argument can be given when $\pi(a') \in \mathfrak S_{s, t}^-$. The only possibility then is that there are no pseudoconcave points near $p_{\infty}$, i.e., $\pa 
D$ is weakly pseudoconvex near $p_{\infty}$. In this case, \cite{SV} shows that 
\begin{equation}
D \backsimeq \ti D = \big\{ (z_1, z_2) \in \mbb C^2 : 2 \Re z_2 + P_{2m}(z_1, \ov z_1) < 0 \big\}
\end{equation}
where $P_{2m}(z_1, \ov z_1)$ is a homogeneous subharmonic polynomial of degree $2m$ -- this being the $1$-type of $\pa D$ at $p_{\infty}$, without harmonic terms. In particular $D$ is 
globally pseudoconvex and as $\pi$ is proper, it follows that $\mathfrak S_{s, t}$ is also pseudoconvex. However, this is not the case.

\medskip

When $0 = s < t < \infty$, the two components of $\pa \mathfrak S_{0, t}$ are 
\begin{align*}
\pa \mathfrak S_{0, t}^+ &= \big\{ (z_1, z_2) \in \mbb C^2 : (\Re z_1)^2 + (\Re z_2)^2 = t \big\}, \; {\rm and} \\
i \mbb R^2               &= \big\{ \Re z_1 = \Re z_2 = 0 \big\}.
\end{align*}
Choose $a \in S$ as above and let $W, W$ be small neighbourhoods of $a$ and $\pi(a)$ so that $\pi : W \ra W'$ is a well defined holomorphic mapping. Suppose that $\pi(a) \in i \mbb 
R^2$. Since $V_{\pi} \cap \pa D$ has real dimension at most one, it follows that there is an open piece of $W \cap \pa D$ near $a$ that is mapped locally 
biholomorphically onto an open piece in $i \mbb R^2$ and this is a contradiction. A similar argument shows that $\pi(a) \notin \pa \mathfrak S_{0, t}^+$ and therefore the only 
possibility is that $\pa D$ is weakly pseudoconvex near $p_{\infty}$. By \cite{SV} it follows that $D \backsimeq \ti D$ where $\ti D$ is as in (3.1). Let $\pi$ still denote the proper 
mapping
\[
\pi : \ti D \ra \mathfrak S_{0, t}.
\]
Let $\phi$ be a holomorphic function on $\ti D$ that peaks at the point at infinity in $\pa \ti D$. Then $\psi = \log \vert \phi - 1 \vert$ is a plurisubharmonic function that is 
bounded above on $\ti D$ and has the property that $\psi \ra -\infty$ at the point at infinity in $\pa \ti D$. If $\pi^{-1}_1, \pi^{-1}_2, \ldots, \pi^{-1}_m$ are the local branches of 
$\pi^{-1}$, then it is known that
\[
\ti \psi = \max \{ \psi \circ \pi^{-1}_j : 1 \le j \le m \}
\]
extends to a plurisubharmonic function on $\mathfrak S_{0, t}$. If there is an open piece of $\pa \mathfrak S_{0, t}^+$ on which $\ti \psi \ra \infty$, then the uniqueness theorem shows 
that $\ti \psi \equiv -\infty$ and this is a contradiction. Thus there is a point, say $p \in \pa \ti D$ whose cluster set under $\pi$ intersects $\pa \mathfrak S_{0, t}^+$. Then $\pi$ 
extends continuously up to $\pa \ti D$ near $p$ and this extension is even locally biholomorphic across strongly pseudoconvex points which are known to be dense on $\pa \ti D$. By 
Webster's theorem, $\pi$ is algebraic. Away from a codimension one algebraic variety $Z$, the inverse $\pi^{-1}$ defines a correspondence that is locally given by finitely many 
holomorphic maps. Since $Z \cap i \mbb R^2$ has real dimension at most one, it is possible to pick $p' \in i \mbb R^2 \sm Z$. The branches of $\pi^{-1}$ will now map an open piece of $i 
\mbb R^2$ near $p'$ locally biholomorphically (shift $p'$ if necessary to achieve this) to an open piece on $\pa \ti D$. This cannot happen as $\pa \ti D$ is not totally real.

\medskip

To conclude, let $f : D \ra \mathfrak S_{s, t}^{(n)}$ be biholomorphic. Since $\mathfrak S_{s, t}^{(n)}$ inherits the complex structure from $\mathfrak S_{s, t}$ via $\Phi_{\cal 
\chi}^{(n)}$, it follows that 
\[
\pi = \Phi_{\cal \chi}^{(n)} \circ f : D \ra \mathfrak S_{s, t}
\]
is an unbranched, proper holomorphic mapping between domains with the standard complex structure. Such a map cannot exist as shown above.
\end{proof}

To construct an infinite sheeted cover of $\mathfrak S_{s, t}$, consider the infinite sheeted covering map
\[
\Phi^{(\infty)}_{\cal \chi} : \mbb C^2 \ra \mbb C^2 \sm \big\{ \Re z_1 = \Re z_2 = 0 \big\}
\]
whose components are given by
\begin{align*}
\ti z_1    &= \exp(\Re z_1) \cos(\Im z_1) + i \Re z_2, \; {\rm and}\\
\ti z_2    &= \exp(\Re z_1) \sin(\Im z_1) + i \Im z_2.
\end{align*}
Equip $\mbb C^2$ with the pull-back complex structure using $\Phi^{(\infty)}_{\cal \chi}$ and denote the resulting complex manifold by $M^{(\infty)}_{\cal \chi}$. For $0 \le s < t < 
\infty$ define
\[
\mathfrak S_{s, t}^{(\infty)} = \big\{ (z_1, z_2) \in M^{(\infty)}_{\cal \chi} : (\ln s)/2 < \Re z_1 < (\ln t)/2 \big\}.
\]
This is seen to be an infinite sheeted covering of $\mathfrak S_{s, t}$, the holomorphic covering map being $\Phi^{(\infty)}_{\cal \chi}$.

\begin{prop}
$D$ is not biholomorphic to $\mathfrak S_{s, t}^{(\infty)}$ for $0 \le s < t < \infty$.
\end{prop}

\begin{proof}
Let $f : D \ra \mathfrak S_{s, t}^{(\infty)}$ be a biholomorphism. Then
\[
\pi = \Phi^{(\infty)}_{\cal \chi} \circ f : D \ra \mathfrak S_{s, t}
\]
is a holomorphic infinite sheeted covering map between domains equipped with the standard complex structure. Using the explicit description of $\Phi^{(\infty)}_{\cal \chi}$, we see that
it maps the boundary of $\mathfrak S_{s, t}^{(\infty)}$ into the boundary of $\mathfrak S_{s, t}$. Hence the cluster set of $\pa D$ under $\pi$ is contained in $\pa \mathfrak S_{s, t}$. 
Now if $0 < s < t < \infty$, then by choosing an appropriate point on the two dimensional stratum $S \subset \cal L$ as in the previous proposition, it follows that $\pa D$ must be 
weakly pseudoconvex near $p_{\infty}$. By \cite{SV}, $D \backsimeq \ti D$ where $\ti D$ is as in (3.1). Hence $\ti D$ covers $\mathfrak S_{s, t}$ and since the Kobayashi metric on $\ti 
D$ is complete, it follows that the same must hold for $\mathfrak S_{s, t}$. Completeness then forces $\mathfrak S_{s, t}$ to be pseudoconvex which it is not. Contradiction.

\medskip

If $0 = s < t < \infty$, then first note that the conclusion that $\ti D$ covers $\mathfrak S_{0, t}$ still holds and let $\pi$ still denote this infinite sheeted covering map. By 
\cite{CP}, there exists a point on $\pa \ti D$ whose cluster set under $\pi$ intersects $\pa \mathfrak S_{0, t}^+$. By standard arguments involving the Kobayashi metric, $\pi$ extends 
continuously up to $\pa \ti D$ near this point. This extension is even locally biholomorphic near strongly pseudoconvex points that are known to be dense in $\pa \ti D$. By Webster's 
theorem, $\pi$ is algebraic and therefore the cardinality of a generic fibre of $\pi$ is finite. This contradicts the fact that $\pi$ is an infinite sheeted cover.
\end{proof}


\subsection{A domain in $\mbb P^2$} Let $\cal Q_+ \subset \mbb C^3$ be the smooth complex analytic set given by
\[
z_0^2 + z_1^2 + z_2^2 = 1.
\]
For $1 \le s < t < \infty$ define
\[
E^{(2)}_{s, t} = \big\{ (z_0, z_1, z_2) \in \mbb C^3 : s < \vert z_0 \vert^2 + \vert z_1 \vert^2 + \vert z_2 \vert^2 < t \big\} \cap \cal Q_+.
\]
This is a two sheeted covering of
\[
E_{s, t} = \big\{ [z_0 : z_1 : z_2] \in \mbb P^2 : s \vert z_0^2 + z_1^2 + z_2^2 \vert < \vert z_0 \vert^2 + \vert z_1 \vert^2 + \vert z_2 \vert^2 < t \vert z_0^2 + z_1^2 + z_2^2 \vert 
\big\},
\]
the covering map being $\psi(z_0, z_1, z_2) = [z_0 : z_1 : z_2]$. Similarly, for $1 < t < \infty$, the map
\[
\psi : E_t^{(2)} \ra E_t
\]
is a two sheeted covering, where
\[
E_t^{(2)} = \big\{ (z_0, z_1, z_2) \in \mbb C^3 : \vert z_0 \vert^2 + \vert z_1 \vert^2 + \vert z_2 \vert^2 < t \big\} \cap \cal Q_+
\]
and
\[
E_t = \big\{ [z_0 : z_1 : z_2] \in \mbb P^2 : \vert z_0 \vert^2 + \vert z_1 \vert^2 + \vert z_2 \vert^2 < t \vert z_0^2 + z_1^2 + z_2^2 \vert \big\}.
\]
To construct a four sheeted cover of $E_{s, t}$, consider the map $\Phi_{\mu} : \mbb C^2 \sm \{0\} \ra \cal Q_+$ whose components are given by
\begin{align*}
\ti z_1 &= -i(z_1^2 + z_2^2) + i(z_1 \ov z_2 - \ov z_1 z_2)/(\vert z_1 \vert^2 + \vert z_2 \vert^2),\\
\ti z_2 &= z_1^2 - z_2^2 - (z_1 \ov z_2 + \ov z_1 z_2)/(\vert z_1 \vert^2 + \vert z_2 \vert^2),\; {\rm and}\\
\ti z_3 &= 2 z_1 z_2 + (\vert z_1 \vert^2 - \vert z_2 \vert^2)/(\vert z_1 \vert^2 + \vert z_2 \vert^2).
\end{align*}
Note that $\Phi_{\mu}$ is a two sheeted cover onto $\cal Q_+ \sm \mbb R^3$. Therefore we may equip the domain of $\Phi_{\mu}$, i.e., $\mbb C^2 \sm \{0\}$ with the pull back complex 
structure using $\Phi_{\mu}$ and denote the resulting complex surface by $M^{(4)}_{\mu}$. For $1 \le s < t < \infty$, the domain
\[
E^{(4)}_{s, t} = \big\{ (z_1, z_2) \in M^{(4)}_{\mu} : ((s-1)/2)^{1/2} < \vert z_1 \vert^2 + \vert z_2 \vert^2 < ((t-1)/2)^{1/2} \big\}
\]
is a four sheeted cover of $E_{s, t}$, the holomorphic covering map being $\psi \circ \Phi_{\mu}$.

\begin{prop}
There cannot exist a proper holomorphic mapping from $D$ onto $E_{s, t}$ for all $1 \le s < t < \infty$. In particular, $D$ is not equivalent to either $E^{(2)}_{s, t}$ or $E^{(4)}_{s, 
t}$.
\end{prop}

\begin{proof}
Let $f : D \ra E_{s, t}$ be a proper holomorphic mapping. Consider the case when $1 < s < t < \infty$. The boundary $\pa E_{s, t}$ has two components, namely
\begin{align*}
\pa E_{s, t}^+ &= \big\{ [z_0: z_1: z_2] \in \mbb P^2 : \vert z_0 \vert^2 + \vert z_1 \vert^2 + \vert z_2 \vert^2 = t \vert z_0^2 + z_1^2 + z_2^2 \vert \big\}, \; {\rm and}\\
\pa E_{s, t}^- &= \big\{ [z_0: z_1: z_2] \in \mbb P^2 : \vert z_0 \vert^2 + \vert z_1 \vert^2 + \vert z_2 \vert^2 = s \vert z_0^2 + z_1^2 + z_2^2 \vert \big\},
\end{align*} 
which are strongly pseudoconvex and strongly pseudoconcave hypersurfaces respectively. The argument used in proposition 3.1 can be applied here to conclude that $p_{\infty} \in \pa D$ 
must be a weakly pseudoconvex point. By \cite{SV} it follows that $D \backsimeq \ti D$ where $\ti D$ is as in (3.1). Thus we have a proper mapping from $\ti D$ onto $E_{s, t}$ which 
implies that $E_{s, t}$ must be holommorphically convex and this is a contradiction.

\medskip

Now suppose that $1 = s < t < \infty$. Then the boundary $\pa E_{1, t}$ consists of a strongly pseudoconvex piece, namely $\pa E_{1, t}^+$ and a maximally totally real piece given by 
$\psi(\pa \mbb B^3 \cap \cal Q_+)$. The argument in the preceeding paragraph applies again to show that $D \backsimeq \ti D$ with $\ti D$ as in (3.1). Let $f$ still denote the proper 
map from $\ti D$ onto $E_{1, t}$. Let $\phi$ be a holomorphic function on $\ti D$ that peaks at the point at infinity in $\pa \ti D$ and denote by $f^{-1}_1, f^{-1}_2, \ldots, f^{-1}_l$ 
the locally defined branches of $f^{-1}$ that exist away from a closed codimension one analytic set in $E_{1, t}$. Then
\[
\ti \phi = (\phi \circ f^{-1}_1) \cdot (\phi \circ f^{-1}_2) \cdots (\phi \circ f^{-1}_l)
\]
is a well defined holomorphic function on $E_{1, t}$ and satisfies $\vert \ti \phi \vert < 1$ there. Now $\ti \phi$ extends across $\psi(\pa \mbb B^3 \cap \cal Q_+)$, which has real 
codimension two and is totally real strata, as well. Thus $\ti \phi \in \cal O(E_t)$ and $\vert \ti \phi \vert \le 1$. If $\vert \ti \phi(a') \vert = 1$ for some $a' \in \psi(\pa \mbb 
B^3 \cap \cal Q_+)$, the maximum principle implies that $\vert \ti \phi \vert \equiv 1$ on $E_{1, t} \subset E_t$ and this is a contradiction. This argument shows that for every $a'\in 
\psi(\pa \mbb B^3 \cap \cal Q_+)$, there is a point $a \in \pa \ti D$ such that the cluster set of $a$ under $f$ contains $a'$. On the other hand, by \cite{CP}, there are points $b, b'$ 
on $\pa \ti D, \pa E_{1, t}^+$ respectively such that the cluster set of $b$ contains $b'$. Thus $f$ will be algebraic by Webster's theorem as before. Away from an algebraic variety $Z 
\subset \mbb P^2$, $f^{-1}$ defines a holomorphic correspondence that locally splits into finitely many holomorphic mappings. Since $Z \cap \psi(\pa \mbb B^3 \cap \cal Q_+)$ has real 
dimension at most one, it is possible to choose $a'\in \psi(\pa \mbb B^3 \cap \cal Q_+) \sm Z$. Now one of the branches of $f^{-1}$ will map $a'$ into $\pa \ti D$ and therefore an 
open piece of the totally real component $\psi(\pa \mbb B^3 \cap \cal Q_+)$ will be mapped locally biholomorphically onto an open piece of $\pa \ti D$. Contradiction. 

\medskip

To conclude, if $D \backsimeq E^{(2)}_{s, t}$ or $E^{(4)}_{s, t}$, then this would imply the existence of an unbranched proper holomorphic mapping from $D$ onto $E_{s, t}$ and this 
cannot happen by the arguments given above.
\end{proof}

\begin{prop}
There cannot exist a proper holomorphic mapping from $D$ onto $E_t$ for all $1 < t < \infty$. In particular, $D$ cannot be equivalent to $E^{(2)}_t$ for all $1 < t < \infty$.
\end{prop}

\begin{proof}
By working in local coordinates it can be seen that $E_t$ is described as a sub-level set of a strongly plurisubharmonic function. Hence $E_t$ must be holomorphically convex and 
therefore $D$ is pseudoconvex if there were to exist a proper map $f : D \ra E_t$. By standard arguments involving the Kobayashi metric, this map $f$ will be continuous up to $\pa D$ 
near $p_{\infty}$. By \cite{CPS} it follows that $p_{\infty}$ is a weakly spherical point on $\pa D$, i.e., there is a defining function for $\pa D$ near $p_{\infty} = 0$ of the form
\[
\rho(z) = 2 \Re z_2 + \vert z_1 \vert^{2m} + \ldots.
\]
Since $p_{\infty}$ is an orbit accumulation point, \cite{SV} shows that $D$ is equivalent to the model domain at $p_{\infty}$, i.e.,
\[
D \backsimeq \big\{ (z_1, z_2) \in \mbb C^2 : 2 \Re z_2 + \vert z_1 \vert^{2m} < 0 \big\}.
\]
This shows that $\dim {\rm Aut}(D) = 4$ which is a contradiction. To conclude, if $D \backsimeq E^{(2)}_t$, then there would exist an unbranched proper mapping from $D$ onto $E_t$ which 
is not possible.
\end{proof}


\subsection{Domains constructed by using an analogue of Rossi's map} For $-1 \le s < t \le 1$ let
\[
\Om_{s, t} = \big\{ (z_1, z_2) \in \mbb C^2 : s \vert z_1^2 + z_2^2 - 1 \vert < \vert z_1 \vert^2 + \vert z_2 \vert^2 - 1 < t \vert z_1^2 + z_2^2 - 1 \vert \big\}
\]
and for $-1 < t < 1$ let
\[
\Om_t = \big\{ (z_1, z_2) \in \mbb C^2 : \vert z_1 \vert^2 + \vert z_2 \vert^2 - 1 < t \vert z_1^2 + \z_2^2 - 1 \vert  \big\}.
\]
It was shown in \cite{I1} that $\Om_t$ has a unique maximally totally real ${\rm Aut}(\Om_t)^c$-orbit, namely
\[
\cal O_5 = \big\{ (\Re z_1, \Re z_2) \in \mbb R^2 : (\Re z_1)^2 + (\Re z_2)^2 < t \big\}
\]
for all $t \in (-1, 1)$. Moreover $\Om_t = \Om_{-1, t} \cup \cal O_5$ for all $t \in (-1, 1)$. 

\medskip

For $1 \le s < t \le \infty$ let
\[
D_{s, t} = \big\{ (z_1, z_2) \in \mbb C^2 : s \vert 1 + z_1^2 - z_2^2 \vert < 1 + \vert z_1 \vert^2 - \vert z_2 \vert^2 < t \vert 1 + z_1^2 - z_2^2 \vert, \Im(z_1(1 + \ov z_2)) > 0 
\big\}
\]
where it is assumed that the domain $D_{s, \infty}$ does not contain the complex curve
\[
\cal O = \big\{ (z_1, z_2) \in \mbb C^2 : 1 + z_1^2 - z_2^2 = 0, \Im(z_1(1 + \ov z_2)) > 0 \big\}.
\]
For $1 \le s < \infty$ let
\[
D_s = \big\{ (z_1, z_2) \in \mbb C^2 : s \vert 1 + z_1^2 - z_2^2 \vert < 1 + \vert z_1 \vert^2 - \vert z_2 \vert^2, \Im(z_1(1 + \ov z_2)) > 0 \big\}
\]
and note that $D_s = D_{s, \infty} \cup \cal O$.

\medskip

Observe that $D$ cannot be equivalent to $\Om_{s, t}$ or $D_{s,t}$ as neither is simply connected. It remains to consider whether $D$ can be equivalent to $\Om_t$ or $D_s$.

\begin{prop}
There cannot exist a proper holomorphic mapping from $D$ onto $\Om_t$ for $-1 < t < 1$ or to $D_s$ for $1 \le s < \infty$.
\end{prop}

\begin{proof}
We first consider $\Om_t$. Let $z_1 = x + iy, z_2 = u + iv$ so that 
\[
\cal O_5 = \big\{ (x, u) \in \mbb R^2 : x^2 + u^2 < 1 \big\}
\]
and its boundary
\[
\pa \cal O_5 = \big\{ (x, u) \in \mbb R^2 : x^2 + u^2 = 1 \big\} \subset \pa \Om_t
\]
for all $t \in (-1, 1)$. Note that $\pa \Om_t \sm \pa \cal O_5$ is a smooth strongly pseudoconvex hypersurface. Suppose that $f : D \ra \Om_t$ is proper. As in proposition 3.1, it is 
possible to choose $a \in S \subset \cal L$ such that $f$ extends holomorphically to a neighbourhood of $a$. By shifting $a \in S$ if necessary we may assume that $f$ is in fact 
locally biholomorphic near $a$. Note that $f(a) \notin \pa \Om_t \sm \pa \cal O_5$, as otherwise there are 
strongly pseudoconcave points near $a$ that will be mapped to strongly pseudoconvex points. The remaining possibility is that $f(a) \in \pa \cal O_5$ which is totally real. Since $f$ is 
locally biholomorphic near $a$, $f$ cannot map an open piece of $\pa D$ near $a$ into $\pa \cal O_5$. Again, there are strongly pseudoconcave points near $a$ that are mapped by $f$ to 
$\pa \Om_t \sm \pa \cal O_5$ which is strongly pseudoconvex and this is a contradiction.

\medskip

Hence the boundary $\pa D$ is weakly pseudoconvex near $p_{\infty}$ and thus $D \backsimeq \ti D$ by \cite{SV} where $\ti D$ is as in (3.1). Let $f : \ti D \ra \Om_t$ still denote the 
biholomorphism. Observe that the automorphism group of $\ti D$ is at least two dimensional; apart from the translations $T_t$, it is also invariant under the one parameter subgroup
\[
S_s(z_1, z_2) = (\exp(s/2m)z_1, \exp(s) z_2),
\]
$s \in \mbb R$. The corresponding real vector fields $X = \Re (i \pa / \pa z_2)$ and $Y = \Re((z_1/2m) \pa / \pa z_1 + z_2 \pa / \pa z_2)$ satisfy $[X, Y] = X$. By the arguments in the 
last part of the proof of proposition 4.1 in \cite{V}, it follows that $D \backsimeq \cal D_4$ where
\[
\cal D_4 = \big\{ (z_1, z_2) \in \mbb C^2 : 2 \Re z_2 + (\Re z_1)^{2m} < 0 \big\}.
\]
Let $f : \cal D_4 \ra \Om_t$ still denote the proper map. Choose an arbitrary strongly pseudoconvex point $b' \in \pa \Om_t \sm \pa \cal O_5$. By \cite{CP} there exists $b \in \pa \cal 
D_4$ such that the cluster set of $b$ under $f$ contains $b'$. Then by well known arguments involving the Kobayashi metric on $\cal D_4$ and $\Om_t$ near $b$ and $b'$ respectively, it 
follows that $f$ is continuous up to $\pa \cal D_4$ near $b$ and $f(b) = b'$. By \cite{CPS}, it follows that $b \in \pa \cal D_4$ must be a weakly spherical point, i.e., there exists a 
coordinate system near $b$ in which the defining equation for $\pa \cal D_4$ is of the form
\[
\rho(z) = 2 \Re z_2 + \vert z_1 \vert^{2m} + \ldots,
\] 
the dots indicating terms of higher order. However, the explicit form of $\pa D_4$ shows that no point on it is weakly spherical.

\medskip

It remains to show that no proper map $f : D \ra D_s$ can exist for $1 \le s < \infty$. Suppose the contrary. Observe that if $s > 1$ then $\pa D_s$ is the disjoint union of three 
components, namely
\begin{align*}
\cal C^1 &= \big\{ 1 + \vert z_1 \vert^2 - \vert z_2 \vert^2 = s \vert 1 + z_1^2 - z_2^2 \vert, \; \Im(z_1(1 + \ov z_2)) > 0 \big\},\\
\cal C^2 &= \big\{ 1 + \vert z_1 \vert^2 - \vert z_2 \vert^2 > s \vert 1 + z_1^2 - z_2^2 \vert, \; \Im(z_1(1 + \ov z_2)) = 0 \big\},\\
\cal C^3 &= \big\{ 1 + \vert z_1 \vert^2 - \vert z_2 \vert^2 = s \vert 1 + z_1^2 - z_2^2 \vert, \; \Im(z_1(1 + \ov z_2)) = 0 \big\}.
\end{align*}
\no Note that $\cal C^1$ is a strongly pseudoconvex hypersurface and that $\Im(z_1(1 + \ov z_2)) = 0$ has an isolated singularity at $(z_1, z_2) = (0, -1)$ away from which it is smooth 
Levi flat. Also, $(0, -1) \notin \cal C^2$ as $s > 1$. As above, choose $a \in S \subset \cal L$ near which $f$ extends locally biholomorphically. Since $\cal C^1$ is strongly 
pseudoconvex, it follows that $f(a) \notin \cal C^1$. Further if $f(a) \in \cal C^2$, then a small open piece of $\pa D$ near $a$ will be mapped locally biholomorphically into the 
Levi flat piece $\big\{ \Im(z_1(1 + z_2)) = 0 \big\}$ and this is a contradiction as points on $\pa D \sm S$ near $a$ are Levi non-degenerate. The remaining possibility is that $f(p) 
\in \cal C^3$. However, an open piece of $\pa D$ near $a$ cannot be mapped by $f$ into $\cal C^3$ as it has real dimension at most $2$ near each of its points. Thus there is an 
open dense set of points near $a$ that are mapped locally biholomorphically into either $\cal C^1$ or $\cal C^2$. Both cannot occur for reasons mentioned above. Thus $\pa D$ must be 
weakly pseudoconvex near $p_{\infty}$ and we may now argue as before to get a contradiction.

\medskip

When $s = 1$, it was noted in \cite{I1} that there is a proper mapping $g$ from the bidisc $\Delta^2$ onto $D_1$. If $f : D \ra D_1$ is proper, then $F : f^{-1} \circ g : \Delta^2 \ra 
D$ is a proper holomorphic correspondence. Thus $D$ is pseudoconvex and by \cite{SV}, it follows that $D \backsimeq \ti D$ where $\ti D$ is as in (3.1). Let $F : \Delta^2 \ra \ti D$ 
still denote the proper correspondence. Using the holomorphic function on $\ti D$ that peaks at the point at infinity in $\pa \ti D$ it can be seen that there is an open dense subset of 
$\pa \Delta^2$ whose cluster set under $F$ intersects the finite part of $\pa \Delta^2$ -- call this subset $\Gamma$. Fix $\z_0 \in \Gamma$ and a small neighbourhood $W$ containing 
it such that $W \cap \pa \Delta^2$ is smooth. Note that $W \cap \pa \Delta^2$ is defined as the zero locus of either $\vert z_1 \vert^2 - 1$ or $\vert z_2 \vert^2 - 1$ both of which are 
plurisubharmonic. Now well known arguments using the branches of $F^{-1}$, these plurisubharmonic defining equations and a suitable version of the Hopf lemma show that
\[
{\rm dist}(F(z), \pa \ti D) \lesssim {\rm dist}(z, \pa \Delta^2)
\]
whenever $z \in W \cap \Delta^2$ -- here $F(z)$ denotes any one of the finitely many branches of $F$. By \cite{BS} it follows that $F$ extends continuously up to $W \cap \pa \Delta^2$ 
as a correspondence. The branching locus of $F$ in $\Delta^2$ is therefore defined by a holomorphic function in $\Delta^2$ that extends continuously up to $W \cap \pa \Delta^2$. Let $h \in 
\cal O(\Delta^2)$ define the branching locus. If $h \equiv 0$ on $W \cap \pa \Delta^2$, the uniqueness theorem shows that $h \equiv 0$ in $\Delta^2$ which cannot happen. By shifting 
$\z_0$ we may assume that $h(\z_0) \not= 0$. Therefore near $\z_0$ the correspondence $F$ splits into well defined holomorphic functions, say $F_1, F_2, \ldots, F_k$ each of which is 
holomorphic on $W$ (shrink $W$ if needed) and continuous up to $W \cap \pa \Delta^2$. Since $W \cap \Delta^2$ is a product domain and each point of $\pa \ti D$ supports a holomorphic 
peak function, arguments from \cite{Ru} show that these branches $F_1, F_2, \ldots, F_k$ must be independent of either $z_1$ or $z_2$. This contradicts the assumption that $F$ is 
proper.
\end{proof}

\no It is also possible to construct finite and infinite sheeted covers of $D_{s, t}, \Om_{s, t}$ as explained in \cite{I1}. That $D$ cannot be equivalent to any of them follows by 
similar arguments and we omit the details.

\medskip

Finally proposition 4.1 of \cite{V} shows that a bounded domain $D \subset \mbb C^2$ that satisfies the hypotheses of the main theorem and admits a Levi flat ${\rm Aut}(D)^c$-orbit must 
be equivalent to
\[
\cal D_4 = \big\{ (z_1, z_2) \in \mbb C^2 : 2 \Re z_2 + (\Re z_1)^{2m} < 0 \big\}.
\]
The proof is purely local and can be applied here as well to conclude that a domain $D \subset X$ as in the main theorem with a Levi flat ${\rm Aut}(D)^c$-orbit must be equivalent to
$\cal D_4$.  This is the only possibility that remains after eliminating all others and the conclusion is that if $\dim {\rm Aut}(D)=3$ then $D \backsimeq \cal D_4$.


\section{Model domains when ${\rm Aut}(D)$ is four dimensional}

\no Of the 7 isomorphism classes listed in \cite{I2} of hyperbolic surfaces with four dimensional automorphism group, the following cannot be equivalent to $D$ for topological reasons.

\begin{itemize}
\item The spherical shell $S_r = \big\{ z \in \mbb C^2 : r < \vert z \vert < 1 \big\}$ for $0 \le r < 1$ -- the automorphism group here is the unitary group $U_2$ which is compact, or 
the  quotient $S_r/ \mbb Z_m$ for some $m \in \mbb N$, none of which are simply connected.
\item $\cal E_{r, \theta} = \big\{ (z_1, z_2) \in \mbb C^2 : \vert z_1 \vert < 1, r(1 - \vert z_1 \vert^2)^{\theta} < \vert z_2 \vert < (1 - \vert z_1 \vert^2)^{\theta} \big\}$, where 
$\theta \ge 0, 0 < r < 1$ or $\theta < 0, r = 0$. This is not simply connected.
\item $D_{r, \theta} = \big\{ (z_1, z_2) \in \mbb C^2 : r \exp(\theta \vert z_1 \vert^2) < \vert z_2 \vert < \exp(\theta \vert z_1 \vert^2) \big\}$, where $\theta = 1, 0 < r < 1$ or 
$\theta = -1, r = 0$. This is again not simply connected.
\end{itemize}

\no The remaining four classes listed below have a common feature that a large part of their boundary, if not the whole, is spherical. 

\begin{itemize}
\item $\Om_{r, \theta} = \big\{ (z_1, z_2) \in \mbb C^2 : \vert z_1 \vert < 1, r(1 - \vert z_1 \vert^2)^{\theta} < \exp(\Re z_2) < (1 - \vert z_1 \vert^2)^{\theta} \big \}$, where 
$\theta = 1, 0 \le r < 1$ or $\theta = -1, r = 0$
\item $\mathfrak S = \big\{ (z_1, z_2) \in \mbb C^2 : -1 + \vert z_1 \vert^2 < \Re z_2 < \vert z_1 \vert^2 \big\}$.
\item $\cal E_{\theta} = \big\{ (z_1, z_2) \in \mbb C^2 : \vert z_1 \vert < 1, \vert z_2 \vert < (1 - \vert z_1 \vert^2)^{\theta} \big\}$, for $\theta < 0$. Here the boundary $\pa \cal 
E_{\theta}$ contains a Levi flat piece $L = \big\{ \vert z_1 \vert = 1 \big\} \times \mbb C_{z_2}$. Away from $L$, $\pa \cal E_{\theta}$ is spherical and strongly pseudoconcave as seen 
from $\cal E_{\theta}$
\item $E_{\theta} = \big\{ (z_1, z_2) \in \mbb C^2 : \vert z_1 \vert^2 + \vert z_2 \vert^{\theta} < 1 \big\}$, where $\theta > 0$ and $\theta \not= 2$. 
\end{itemize}

To see that $D$ cannot be equivalent to $\Om_{r, \theta}, \mathfrak S$ or to $\cal E_{\theta}$, suppose the contrary. Let $f : D \ra \mathcal E_{\theta}$ be biholomorphic. Let $p \in 
\pa D$ be a strongly pseudoconcave point near $p_{\infty}$ across which $f$ extends locally biholomorphically. Note that $f(p) \notin L$ as $\pa D$ is of finite type near $p_{\infty}$. 
Then $f(p) \in \pa \cal E_{\theta}$. Let $g$ be a local biholomorphism defined on a open neighbourhood $W$ of $f(p)$ that takes $W \cap \pa \cal E_{\theta}$ into $\pa \mbb B^2$. Then 
$g \circ f$ is a biholomorphic germ at $p$ that maps an open piece of $\pa D$ into $\pa \mbb B^2$. By \cite{Sh}, this germ can be analytically continued along all paths in $U \cap \pa 
D$ that start at $p$. Thus $p_{\infty}$ must be a weakly pseudoconvex point and by \cite{CPS}, it must be weakly spherical as well. By \cite{SV}, it follows that $D \backsimeq E_{2m}$ 
and so $\cal E_{\theta} \backsimeq E_{2m}$ which is a contradiction.

\medskip

To conclude, it remains to show that if $D \backsimeq E_{\theta}$, then $\theta = 2m$ for some integer $m \ge 2$. Proposition 5.1 in \cite{V} remains valid here too and we omit the 
details. The conclusion is that if $\dim {\rm Aut}(D) = 4$ then $D \backsimeq E_{2m} \backsimeq \cal D_5$.


\end{document}